\documentclass[11pt]{article}
\usepackage{graphicx}
\usepackage{amsmath}
\usepackage{amsfonts}
\usepackage[english]{babel}
\usepackage{amsthm}
\usepackage{latexsym,color}
\numberwithin{equation}{section}
\newtheorem{theorem}{Theorem}[section]
\newtheorem{corollary}[theorem]{Corollary}

\newtheorem{proposition}[theorem]{Proposition}
\newtheorem{remark}[theorem]{Remark}
\def\neweq#1{\begin{equation}\label{#1}}
\def\endeq{\end{equation}}

\input texdraw

\newcommand{\R}{\mathbb{R}}

\renewcommand{\div}{{\rm div}}
\newcommand{\om}{\Omega}
\newcommand{\ome}{\omega_\epsilon}
\newcommand{\omen}{\omega_{\epsilon_n}}

\newcommand{\eps}{\epsilon}

\begin{document}

\title{On a semilinear elliptic boundary value problem arising in cardiac electrophysiology}

%\author{Elena BERETTA -- M. Cristina CERUTTI -- Andrea MANZONI -- Dario PIEROTTI\\
%{\small Dipartimento di Matematica del Politecnico, Piazza L. da Vinci 32 -
%20133 Milano (Italy)}}
\date{}

\author{Elena Beretta\thanks{Dipartimento di Matematica "F. Brioschi", Politecnico di Milano" ({\tt elena.beretta@polimi.it})}
\and M.Cristina Cerutti\thanks{Dipartimento di Matematica "F. Brioschi", Politecnico di Milano" ({\tt cristina.cerutti@polimi.it})}
\and Andrea Manzoni\thanks{CMCS-MATHICSE-SB, Ecole Polytechnique F\'ed\'erale de Lausanne ({\tt andrea.manzoni@epfl.ch})}
\and Dario Pierotti\thanks{Dipartimento di Matematica "F. Brioschi", Politecnico di Milano" ({\tt dario.pierotti@polimi.it})}
}
\maketitle

\begin{abstract}
In this paper we provide a representation formula for boundary voltage perturbations caused by internal conductivity inhomogeneities of low volume fraction in a simplified {\em monodomain model} describing the electric activity of the heart. We derive such a result in the case of a nonlinear problem. Our long-term goal is the solution of the inverse problem related to the detection of regions affected by heart ischemic disease, whose position and size are unknown. We model the presence of ischemic regions in the form of small  inhomogeneities. This leads to the study of a boundary value problem for a semilinear elliptic equation. We first analyze the well-posedness of the problem establishing some  key energy estimates. These allow us  to derive rigorously an asymptotic formula of the boundary potential perturbation due to the presence of the inhomogeneities, following an approach similar to the one introduced by Capdeboscq and Vogelius in \cite{capvoge}  in the case of  the linear conductivity equation. Finally, we propose some ideas of the reconstruction procedure that might be used to detect the inhomogeneities.

\end{abstract}

\section{Introduction}

Ischemic heart disease results from a restriction in blood supply to the heart and represents  the most widespread heart disease. As a consequence, myocardial infarction (or heart attack) caused by the lack of oxygen might lead to even more severe heart muscle damages, ventricular arythmia and fibrillation, ultimately causing death. Detecting ischemic heart diseases -- that is, recovering the unknown shape (and/or position) of ischemic areas -- at early stages of their development from noninvasive (or minimally invasive) measurements is thus of primary importance.

%%%%%%%% OLD VERSION %%%%%%%%%%%%%
%Ischemic heart disease is the most widespread heart disease and one of the most common causes of death in the world. Also known as coronary artery disease or atherosclerotic heart disease, it is the result of a restriction in blood supply to the heart.

%This latter is usually caused by the formation of plaques within the inner walls of coronary arteries, and their subsequent damage. As a matter of fact, a strong limitation of blood supply implies a lack of oxygen, then cell starvation (myocardial infarction, or heart attack) and subsequently to severe heart muscle damages, later myocardial scarring without heart muscle regrowth.(\cite{} reference for ischemia)

%%Another form of ischemic heart disease caused by severe stenosis of coronary arteries is transient ischemia. This might lead to ventricular arrhythmia and fibrillation, ultimately causing death. Detecting ischemic heart diseases at early stages of their development is thus of primary importance.

This is usually performed by recording the electrical activity of the heart, by means of either body surface or intracardiac measurements. In the former case, the electrocardiogram (ECG) records electrical impulses  %- generated by  polarization and depolarization stages of the cardiac tissue --
across the thorax, by means of a set of electrodes attached to the surface of the skin. In the latter case, intracardiac electrograms, that is, measurements of intracavitary potentials, are obtained  by means of non-contact electrodes carried by a catheter inside a heart cavity.
Although much more invasive than ECG, this latter technique has become a standard of care in patients with symptoms of heart failure, and allows to get a map of the endocardial potential.

In this context, mathematical models could be used to shed light on the potentialities of electrical measurements in detecting ischemia. More specifically, the goal would be to combine measurements of (body-surface or intracavitary) potentials and a mathematical description of the electrical activity of the heart in order to identify the position, the shape and the size of heart ischemias and/or infarctions.
 It is well known (see e.g. \cite{SLCNMT,Pavarino_2014_book}), that a mathematical description of  the electrical activity of the heart is provided by the so-called {\it bidomain model}, yielding to an initial boundary value problem for a coupled nonlinear evolution system. A simplified, one-field version of this problem is provided by the monodomain model, resulting in a nonlinear diffusion-reaction equation. %The study of this problem is rather complicated from a theoretical (well-posedness)  viewpoint.\cite{}
Moreover, the myocardium is surrounded by a volume conductor, the {\em torso}, which is commonly modelled as a passive conductor through a linear elliptic equation; heart and torso are coupled by imposing the continuity of the electrical potential and the currents across the interface.

The challenge of how to combine ECG recordings or intracavitary potential measurements with mathematics and computations to identify ischemic heart disease has by far not been investigated enough. Although some  analysis of the direct problem  has been carried out, -- see, e.g.   \cite{Gerbeau_2008} -- 
and such a model (coupled with the torso) has been exploited for the ultimate generation of synthetic ECG data  \cite{BCFGZ}, to our knowledge there is no theoretical investigation of  inverse problems related with ischemia detection involving the monodomain and/or the bidomain model, not even in the case of an isolated heart.
%(that is, without taking into account the heart-torso coupling).} % for the description of the electrical activity of the heart.
%The challenge of how to combine ECG recordings or intracavitary potential measurements with mathematics and computations to identify ischemic heart disease by far has not been investigated enough.
In the past decade some numerical investigations dealing with ischemia identification from measurements of surface potentials have been performed by casting the problem in an optimization framework. A stationary model  taking into account the heart-torso coupling has been employed in \cite{Nielsen_2007}, whereas  a nonstationary monodomain model for an isolated heart has  been considered in \cite{LN}. More recently, the case of ischemias identification from intracardiac electrograms  has been treated in \cite{Alvarez_2012}.
The question of finding the ischemic region can be formulated as the inverse problem of detecting inhomogeneities, whose position and size are unknown, in a nonlinear parabolic diffusion-reaction PDE, modeling (for the time being, a much simplified version of) the cardiac electrical activity from boundary measurements\footnote{ \textcolor{black}{We point out that this is completely different from  what is commonly referred to as the {\em inverse problem of electrocardiography} which deals with recovering the electrical potential at the epicardial surface by using recordings of the electrical potential along the body surface,  \cite{MacLeod_2010}, \cite{Colli_Franzone_1978}, \cite{Yamashita_1982} and which involves  the pure (linear)  diffusion model for the torso as direct problem.}} %In this context, the inverse boundary value problem from partial Cauchy data is severely ill-posed (see e.g. \cite{Colli_Franzone_1978,Yamashita_1982}).}}
%A first numerical investigation of this has been carried out in \cite{LN} where some preliminary results concerning the inverse problem for the  {\it monodomain model} have been obtained using level-set techniques. %%The challenge of how to combine ECG recordings with mathematics and computations to identify ischemic heart disease by far has not been investigated enough.
%In \cite{} the impact of infarctions on ECG signals has been addressed from a numerical standpoint.
%
%
 In the present paper we assume
 %\footnote{\textcolor{red}{The analysis of the more general model involving the heart-torso coupling, as well as the possibility to perform body surface measurements, will be tackled in the near future.}}
 to be able to perform measurements on the heart (by one of the devices described above) and use an insulated {\it monodomain model}  in the steady state.  This leads to the study of a Neumann boundary value problem for a semilinear elliptic equation.
%\begin{equation}\label{eq:problem:initial1}
%\begin{array}{rll}
 %\displaystyle     -   \nabla \cdot(k_\eps \nabla u)   + \chi_{\om\setminus\ome} I_{ion}(u)  = & f
%& \mbox{in } \Omega  \\[8pt]
 %\displaystyle  \frac{\partial u}{\partial{\bf n}} = & 0 & \mbox{on }  \partial \Omega  \\[5pt]
   %u({\bf x}, 0) =  & u_0({\bf x}) & \mbox{in}  \ \Omega.
 %  \end{array}
%\end{equation}
%where $\Omega$ is the domain occupied by the heart, $u$ is the (transmembrane) electrical potential and the ionic current  $I_{ion}$ across the cell membrane is assumed to be a nonlinear function of the potential; see Section 2 for a complete description of problem \eqref{eq:problem:initial1}.

We assume that the ischemic region is a small inclusion $\omega_{\eps}$  with a significantly  different conductivity from the healthy tissue.

Taking advantage of the smallness of the inclusion, we establish a rigorous asymptotic expansion of the boundary potential perturbation due to the presence of the inclusion following  the approach introduced by Capdeboscq and Vogelius in \cite{capvoge} in the case of the linear conductivity equation. It turns out that this approach has been successfully used for the reconstruction of location and geometrical features of the unknown inclusions from boundary  measurements \cite{AMV,AK} in the framework of  Electrical Impedance Tomography (EIT) imaging techniques.
Despite of the fact that we have to deal  with a nonlinear equation, we derive a rigorous expansion for the perturbed electrical potential and give also some idea of the reconstruction procedure that might be used to detect the inclusion.\\

The paper is organized as follows: in  Section 2 we illustrate the monodomain model for the cardiac electric activity. In Section 3 we state our main result. In Section 4 we analyze the wellposedness of the direct problem establishing some key energy estimates. In Section 5 we derive the asymptotic formula for the electrical boundary potential. In Section 6, taking advantage of the asymptotic formula, we highlight some ideas for the reconstruction algorithm in a {simplified two-dimensional geometry.

\section{The direct problem: a nonlinear diffusion-reaction equation}

%The transmembrane  potential (or action potential) is the potential jump $v$ across the cellular membrane surface, and it constitutes the physical quantity of interest when we aim at modeling the bioelectrical activity of a cardiac myocite (or muscle cell). Such a mechanism is first described at the cellular level, then extended to the entire tissue based on a volume-averaged approach.

%The mathematical model for the description of cardiac electrophysiology is the so-called {\em bidomain model} \cite{}: it

The bidomain equations are nowadays  the most widely accepted mathematical model of the macroscopic electrical activity of the heart \cite{SLCNMT,Pavarino_2014_book}. This model describes the evolution of the  transmembrane  potential (or action potential), that is,   the potential jump $u$ across the cellular membrane surface. Such a model  is based on the assumption that, at the cellular scale, cardiac tissue is considered as partitioned in two conducting media, the intracellular (made of cardiac cells) and the extracellular (representing the space between cells) medium, separated by the cell membrane. After a homogenization process, the two media are supposed to occupy the whole heart volume.

The bidomain model consists of a coupled system of time-dependent, nonlinear reaction-diffusion PDEs, whose efficient numerical solution is a very difficult task. Since our main goal   is  to provide a theoretical analysis of the inverse problem related with the detection of small conductivity inhomogeneities for a nonlinear diffusion-reaction equation, we rather start from the simpler
%-- but less physically accurate --
{\em monodomain model}.
% which shows  these features.
The  monodomain model   is derived from the bidomain equations  by assuming that the extracellular and the intracellular conductivities are proportional quantities.
The monodomain model for the electrical activity in the heart reads as follows:
\begin{equation}\label{eq:problem:initial1}
\begin{array}{rll}
 \displaystyle    \frac{\partial u}{\partial t}  -   \div(k_{\epsilon}  \nabla u)   + \chi_{\Omega\setminus \omega_\epsilon} I_{ion}(u)  = & f
& \mbox{in} \ Q_T = \Omega \times (0,T] \\[8pt]
  \frac{\partial u}{\partial {\bf n}} = & 0 & \mbox{on} \  \partial \Omega \times (0,T] \\[5pt]
   u({\bf x}, 0) =  & u_0({\bf x}) & \mbox{in}  \ \Omega.
   \end{array}
\end{equation}
where $\Omega$ is the domain occupied by the heart, $u$ is the (transmembrane) electrical potential, $I_{ion}$ is the ionic membrane current of the heart tissue (up to a capacity constant), $k_{\epsilon}$ is its conductivity, $f$ is an applied current (up to the same capacity constant as $I_{ion}$)  and $\chi_{\Omega\setminus \omega_\epsilon}$ is the characteristic function of the healthy area.
%$I_{ion}$ is the ionic membrane current of the heart tissue (up to a capacity constant), $k$ is its conductivity, $I_{app}$ is an applied current (up to the same capacity constant as $I_{ion}$) and $\chi_{\om\setminus\ome}$ is the characteristic function of the healthy (i.e. not infarcted) area
 %-- that is,  a ``scaled'' conductivity.
 %Typically, the ionic current  $I_{ion}$ across the cell membrane is assumed to be a nonlinear function of the potential and a significant choice is to assume  $I_{ion}$ to be a cubic polynomial in $u$  i
 Here $\omega_{\epsilon} \subset \Omega$ is the infarcted area. According to experimental observations, in an ischemic or infarcted region cells are not excitable, so that the conductivity $k_{\epsilon}=k_{\epsilon}({\bf x})$ is substantially different with respect to healthy tissues. For this reason, we define
\begin{equation}\label{eq:problem:initial2}
 k_{\epsilon}  = \left\{
 \begin{array}{rl}
 k_{healthy} & \mbox{in } \ \Omega \setminus \omega_{\epsilon} \\
 k_{infarcted} & \mbox{in } \  \omega_{\epsilon}
 \end{array}
 \right.
\end{equation}
being $k_{infarcted} = \delta k_{healthy}$, and $\delta \in (0,1)$. Moreover the ion transport circumvent ischemic areas, so that also the   ionic membrane current $I_{ion}$ is multiplied by  $ \chi_{\Omega\setminus \omega_\epsilon}$ in order to describe a blocking ion transport.
Typically, the ionic current  $I_{ion}$ across the cell membrane is assumed to be a nonlinear function of the potential and a significant choice is to assume  $I_{ion}$ to be a cubic polynomial in $u$
such as
  \begin{equation}\label{eq:problem:initial3}
I_{ion}(u) = A^2(u-u_{rest})(u-u_{th})(u-u_{peak}).
 \end{equation}
 Here   $A>0$ is a parameter determining the rate of change of $u$ in the depolarization phase, and $u_{peak} >  u_{th} > u_{rest}$ are given constant values representing the resting, threshold and maximum potentials, which affect the action potential dynamics. The definition of $I_{ion}$ depends on the considered cell membrane ionic model (see e.g. \cite{SLCNMT,Pavarino_2014_book,LN} and references therein), and ultimately on intra- and extracellular ionic concentration.  %and the so-called gating variables related to the ionic flow.
  A very large amount of ionic models have been proposed; here we adopt the simplest {\em phenomenological model}, which does not involve any further ionic variable.

  Finally, $f$ represents a given current stimulus applied to the tissue -- usually in a confined region and for a short time interval -- expressing  the initial electrical stimulus, related to the so-called {\em pacemaker potential}. By solving problem \eqref{eq:problem:initial1}  we can describe the propagation of the stimulus $f$ in an insulated heart muscle, affected by ischemia in the region  $\omega_{\epsilon}$. Changing the size and the location of $\omega_{\epsilon}$ thus results in a different propagation of the applied current. %a more realistic situation would involve the coupling with a model governing the potential in the {\em torso} -- i.e., the region outside the heart.   \\

Starting from model \eqref{eq:problem:initial1}--\eqref{eq:problem:initial3}, we consider some simplifications in the direct problem and hence also in the corresponding inverse problem. In fact, we believe that the key aspect to be tackled is related with the presence of a nonlinear term in the equation.
First, we neglect anisotropic description of the electrical conductivity over the domain, we replace the term $I_{ion} = I_{ion}(u)$ by the cubic nonlinearity $\tilde{I}_{ion}(u) = u^3$ and finally we consider the steady version of the problem. The analysis of the more general problem will be addressed in the future.
%
%Then, we replace the term $I_{ion} = I_{ion}(u)$ by the cubic nonlinearity $\tilde{I}_{ion}(u) = u^3$, since in the  range $(v_{rest},v_{peak})$ the two functions show the same monotonicity properties. Then, we consider the steady version of the problem at the moment -- the analysis of the time-dependent case will be addressed in the future.
%
%Then, we replace the term $I_{ion} = I_{ion}(u)$ by the cubic nonlinearity $\tilde{I}_{ion}(u) = u^3$, since the mathematical difficulties related with the proper handling of nonlinear cubic terms in ... is ... (substantially the same?) %in the  range $(v_{rest},v_{peak})$ the two functions show the same monotonicity properties.
%Finally, we consider the steady version of the problem at the moment -- the analysis of the time-dependent case will be addressed in the future.}

\section{Statement of the problem and main result}

As discussed in the previous section, the problem of determining a small `inhomogeneity' $\ome$ inside a smooth domain $\om$, meaning a subset in which the conductivity is smaller then in the surrounding tissue leads to solving the following problem for the potential that here will be called $u_\epsilon$

\begin{equation}
\label{probeps}
\left\{
  \begin{array}{ll}
    -\div(k_\eps(x)\nabla u_\eps)+\chi_{\om\setminus\ome}u_\eps^3=f, & \hbox{in $\om$} \\
    \displaystyle{\frac{\partial u_\eps}{\partial\mathbf{n}}}=0, & \hbox{on $\partial\om$},
  \end{array}
\right.
\end{equation}

where $\om\subset \R^N$, $N=2,\,3$ and $\ome\subset\om$ is the set of inhomogeneity that we assume to be measurable and separated from the boundary
of $\om$, meaning that there exist a compact set $K_0$ with $\ome\subset K_0\subset\om$ and $d_0>0$ such that
\begin{equation}
\label{sepbound}
{\rm dist}(\ome,\om\setminus K_0)\geq d_0>0
\end{equation}
Moreover $|\ome|>0$ $\forall\eps$ and $|\ome|\to 0$ as $\eps\to0$. By denoting with $\mathbf{1}_{\ome}$ the indicator function of the set $\ome$, it is known that there exist a regular Borel measure $\mu$ and a sequence $\omega_{\eps_n}$, with $|\omega_{\eps_n}|\to 0$, such that
\begin{equation}
\label{defmu}
|\omega_{\eps_n}|^{-1}\mathbf{1}_{\omega_{\eps_n}}\,dx\rightarrow d\mu
\end{equation}
in the weak$^*$ topology of the dual of $\mathcal{C^0}(\bar\om)$ (see, e.g. \cite{brezis}). Moreover, $\mu$ is a probability measure and by \eqref{sepbound} its support lies inside the compact set $K_0$.

The function $k_\eps(x)$ represents the conductivity in the two portions of $\om$
and is defined as

\begin{equation}
\label{defkesp}
k_\eps=\left\{
  \begin{array}{ll}
    1, & \hbox{in $\om\setminus\ome$} \\
    k, & \hbox{in $\ome$,}
  \end{array}
\right.
\end{equation}

where we assume $0<k<1$.

The potential $U$ for the unperturbed problem satisfies

\begin{equation}
\label{probu}
\left\{
  \begin{array}{ll}
    -\Delta U+U^3=f, & \hbox{in $\om$} \\
    \displaystyle{\frac{\partial U}{\partial\mathbf{n}}}=0, & \hbox{on $\partial\om$},
  \end{array}
\right.
\end{equation}

For any given $U\in \mathcal{C}^1(\overline\om)$ we introduce the Green function $N_U(x,y)$ for the operator $-\Delta+3U^2$ with \emph{homogeneous} Neumann condition:
\begin{equation}\label{greenU}
-\Delta_x N_U(x,y)+3U^2(x)N_U(x,y)=\delta(x-y)\quad \mathrm{for}\,\, x\in\om,\quad\quad\quad \frac{\partial N_U}{\partial n_x}\Big |_{\partial\om}=0
\end{equation}

We are now ready to state our main result

\begin{theorem}
\label{asyrepr}
Let $f\in L^p(\om)$ for some $p>N$ and assume that $f(x)\ge m>0$ a.e.in $\om$. Let $u_{\eps}$, $U$ denote the solutions to \eqref{probeps} and \eqref{probu}. Then there exists a sequence $\omega_{\eps_n}$ with $|{\omega}_{\eps_n}|\to 0$ and satisfying \eqref{sepbound}, \eqref{defmu}, such that, if $w_{\eps_n}=u_{\eps_n}-U$,
\begin{equation}
\label{asyformfinal}
w_{\eps_n}(y)=|\omen|\int_{\om} \Big ((1-k)\mathcal{M}_{i\,j}\frac{\partial U}{\partial x_i}\frac{\partial N_U}{\partial x_j}+U^3 N_U\Big )\,d\mu(x) +o(|\omen|)\quad\quad y\in\partial \om\,,
\end{equation}
where $N_U(x,y)$ is the solution of \eqref{greenU} and $\mathcal{M}_{i\,j}\in L^2(\om,d\mu)$. Moreover, the values
$\mathcal{M}_{i\,j}(x)$ satisfy
\begin{equation}\label{mij}
\begin{array}{rll}
\mathcal{M}_{i\,j}(x)=&\mathcal{M}_{j\,i}(x)\quad \mathrm{and}
 &  |\xi|^2\le\mathcal{M}_{i\,j}(x)\xi_i\xi_j\le \frac{1}{k}|\xi|^2\, ,
   \end{array}
\end{equation}
$\xi\in \R^N$, $\mu$ almost everywhere in $\om$.
\end{theorem}

\bigskip

\section{The direct problem}

\begin{theorem}
\label{exist}
\label{ext} Assume that $f\in H^{-1}(\om)$, the dual space of $H^1(\om)$.
Then problem \eqref{probeps} and \eqref{probu}  have a unique solution $u_\eps\in H^1(\om)$,  $U\in H^1(\om)$ respectively.
\end{theorem}

\begin{proof}
By multiplying the equation in \eqref{probeps} by a test function $\phi$, integrating by parts and using the boundary Neumann condition, we obtain the weak formulation

\begin{equation}
\label{weakfuep}
\int_\om k_\eps\nabla u\nabla\phi + \int_{\om\setminus\ome}u^3\phi=\int_\om f\phi \ \ \ , \ \ \ \forall\phi\in H^1(\om).
\end{equation}
Now let $T:H^1(\om)\longrightarrow H^{-1}(\om)$ be the operator defined by
\begin{equation}
\nonumber
\langle Tu,\phi\rangle =\int_\om k_\eps\nabla u\nabla\phi + \int_{\om\setminus\ome}u^3\phi,\quad\quad\forall \phi\in H^1(\om)
\end{equation}
It is readily verified that $T$ is a \emph{potential operator}, that is
$Tu-f$ is the derivative of the functional
\begin{equation}
\label{funz}
E(u)=\frac{1}{2}\int_\om k_\eps|\nabla u|^2+\frac{1}{4}\int_{\om\setminus\ome}u^4-\int_\om fu
\end{equation}
Then, the theorem will follow by showing that $T$ is \emph{bounded, strictly monotone and coercive}; in fact, by these properties of $T$ the functional $E$ is coercive and weakly lower semicontinuous on $H^1(\om)$ (see e.g. \cite{fucik}, theorem 26.11). Thus, $E$ is bounded from below and attains its infimum at some $u_{\eps}\in H^1(\om)$ satisfying
$Tu_{\eps}=f$. The uniqueness of $u_{\eps}$ is a consequence of the strict monotonicity of $T$; for, if $Tu=Tv=f$, equation \eqref{strictmon} below implies $u=v$.

\bigskip

{\bf i.} $T$ is bounded.

\bigskip

By H\"older's inequality
\begin{equation}
\nonumber
|\langle Tu,\phi\rangle|\leq\|\nabla u\|_{L^2(\om)}\|\nabla\phi\|_{L^2(\om)} + \|u\|_{L^6(\om)}^3\|\phi\|_{L^2(\om)}
\end{equation}

and by Sobolev embedding theorem  $\|u\|_{L^6(\om)}\leq C_S\|u\|_{H^1(\om)}$, so that

\begin{equation}
\nonumber
|\langle Tu,\phi\rangle|\leq\|\nabla u\|_{L^2(\om)}\|\nabla\phi\|_{L^2(\om)} + C^3_S\|u\|^3_{H^1(\om)}\|\phi\|_{L^2(\om)}\leq\max\left[\|u\|_{H^1(\om)},C^3_S\|u\|_{H^1(\om)}^3\right]\|\phi\|_{H^1(\om)}.
\end{equation}

Therefore, if $u$ belongs to a bounded subset of $H^1(\om)$,

\begin{equation}
\nonumber
\|Tu\|_{H^{-1}(\om)}=\sup_{\phi}\frac{|\langle Tu,\phi\rangle|}{\|\phi\|_{H^1(\om)}}\leq\max\left[\|u\|_{H^1(\om)},C^3_S\|u\|_{H^1(\om)}^3\right]=C_2.
\end{equation}

\bigskip

{\bf ii.} $T$ is (strictly) monotone.

\bigskip

\begin{equation}
\nonumber
\langle Tu-Tv,u-v\rangle=\int_\om k_\eps|\nabla (u-v)|^2 + \int_{\om\setminus\ome} (u-v)^2(u^2+uv+v^2)\geq0.
\end{equation}

Furthermore

\begin{equation}
\label{strictmon}
\langle Tu-Tv,u-v\rangle=0 \ \ \ \ \ \Leftrightarrow \ \ \ u=v.
\end{equation}

\bigskip

{\bf iii.} $T$ is coercive, that is
\begin{equation}
\label{coerc}
\lim_{\|u\|_{H^1(\om)}\to +\infty}\frac{\langle Tu,u\rangle}{\|u\|_{ H^1(\om)}}=+\infty
\end{equation}

By using again H\"older's inequality,

\begin{equation}
\nonumber
\langle Tu,u\rangle\geq k\int_\om |\nabla u|^2 + \int_{\om\setminus\ome} u^4\geq k\|\nabla u\|^2_{L^2(\om)}+\frac{1}{|\om\setminus\ome|}
\left(\int_{\om\setminus\ome}u^2\right)^2\geq k \|\nabla u\|^2_{L^2(\om)}+\frac{1}{|\om|}
\|u\|^4_{L^2(\om\setminus\ome)}=
\end{equation}

\begin{equation}
\nonumber
=k \left(\|\nabla u\|^2_{L^2(\om)}+\|u\|^2_{L^2(\om\setminus\ome)}\right)+\frac{1}{|\om|}\;
\|u\|^4_{L^2(\om\setminus\ome)}-k\|u\|^2_{L^2(\om\setminus\ome)}.
\end{equation}
\medskip
Finally, by the Poincar\'e inequality (see Appendix) and by $|\om|^{-1}x^4-kx^2 \ge -k^2|\om|/4$,  we get
\begin{equation}
\label{stimacoerc}
\langle Tu,u\rangle\geq k C\|u\|^2_{H^1(\om)}-\frac{k^2}{4}|\om|
\end{equation}
for some positive constant $C$; hence, \eqref{coerc} follows.
\end{proof}

\begin{remark}
\label{positive}
If $f$ is positive \big (i.e. $\langle f,\phi\rangle\ge 0$ for $\phi\ge 0$\big ) it follows from \eqref{funz} that $E(|u|)\le E(u)$ for every $u\in H^1(\om)$; on the other hand, we proved in the previous theorem that
$u_{\eps}$ is the unique minimum of $E$ in $H^1(\om)$. Then, we  conclude that $u_{\eps}\ge 0$.
\end{remark}
\begin{remark}
An alternative proof of theorem \ref{exist} can be obtained from the Minty-Browder theorem (see \cite{brezis} theorem 5.16))
by showing that the (monotone, coercive) non linear operator $T$ is \emph{continuous}. In fact, for $N\le 3$ we have by H\"{o}lder inequality
\begin{equation}
\nonumber
|\langle Tu-Tu_0,\phi\rangle| =\Big |\int_\om k_\eps\nabla (u-u_0)\nabla\phi + \int_{\om\setminus\ome}(u-u_0)(u_0^2+u_0 u+u^2)\phi\Big |
\end{equation}
\begin{equation}
\nonumber
\le \|\nabla (u-u_0)\|_{L^2(\om)}\|\nabla\phi\|_{L^2(\om)} + \|u-u_0\|_{L^6(\om)}\|u_0^2+u_0 u+u^2\|_{L^3(\om)}\|\phi\|_{L^2(\om)}
\end{equation}
for every $u_0$, $u$, $\phi$ in $H^1(\om)$.
Hence, by the Sobolev imbedding $H^1(\om)\hookrightarrow L^6(\om)$ we find that for every $u$, $u_0$ in a bounded subset of $H^1(\om)$ there exist a positive constant $K$ such that
\begin{equation}
\nonumber
|\langle Tu-Tu_0,\phi\rangle|\le K\|u-u_0 \|_{H^1(\om)}\|\phi\|_{H^1(\om)},\quad\quad \forall\,\phi\in H^1(\om)
\end{equation}
It follows that $T$ is locally Lipschitz continuous.
\end{remark}

\subsection{Main estimates}

In this section we will prove estimates on the solutions to \ref{probeps} which will be useful in the subsequent discussion. To begin with, we have the following bound:
%which follows directly from the estimates of theorem \ref{exist}

\begin{proposition}
\label{mainest}
Let $u_{\eps}\in H^1(\om)$ be a solution of \ref{probeps}. Then
\begin{equation}
\label{apriori}
%\|u_{\eps}\|_{H^1(\om)}\le \frac{1}{kC}\,\|f\|_{H^{-1}}+\frac{1}{C^{2/3}}|\om |^{1/3}\|f\|^{1/3}_{H^{-1}}
\|u_{\eps}\|_{H^1(\om)}\le C(\|f\|_{H^{-1}}+\|f\|^{3}_{H^{-1}})
\end{equation}
where the constant $C=C(\Omega,k)$.
\end{proposition}
\begin{proof}
By putting $\phi=u=u_{\eps}$ in equation \eqref{weakfuep} and by definition \eqref{defkesp}, we readily get
\begin{equation}
\label{stim1}
k\|\nabla u_{\eps}\|^2_{L^2(\om)}+\int_{\om\setminus\ome}u^4\le\|f\|_{H^{-1}}\|u_{\eps}\|_{ H^1(\om)}
\end{equation}
By the above inequality we first obtain
\begin{equation}
\nonumber
\|\nabla u_{\eps}\|^2_{L^2(\om)}\le \frac{\|f\|_{H^{-1}}}{k}\|u_{\eps}\|_{ H^1(\om)}
\end{equation}
Furthermore, by the inequality
\begin{equation}
\nonumber
\|u_{\eps}\|^4_{L^2(\om\setminus\ome)}\le |\om\setminus\ome|\,\int_{\om\setminus\ome}u^4 \le |\om|\,\int_{\om\setminus\ome}u^4
\end{equation}
and again by \eqref{stim1} we get
\begin{equation}
\nonumber
\|u_{\eps}\|^2_{L^2(\om\setminus\ome)}\le \big (|\om|\,\|f\|_{H^{-1}}\|u_{\eps}\|_{ H^1(\om)}\big )^{1/2}
\end{equation}
Then, by using again the  Poincar\'e inequality
\begin{equation}
\label{estint}
\|u_{\eps}\|^2_{H^1(\om)}\le \frac{1}{C}\Big (\|\nabla u_{\eps}\|^2_{L^2(\om)}+ \|u_{\eps}\|^2_{L^2(\om\setminus\ome)}\Big )
\le \frac{1}{kC}{\|f\|_{H^{-1}}}\|u_{\eps}\|_{ H^1(\om)}+\frac{1}{C}{\big (|\om|\,\|f\|_{H^{-1}}\big )^{1/2}}\|u_{\eps}\|^{1/2}_{ H^1(\om)}
\end{equation}
We can write the above estimate in the form
\begin{equation}
\nonumber
\|u_{\eps}\|^{1/2}_{H^1} \Big (\|u_{\eps}\|_{H^1}-\frac{1}{kC}{\|f\|_{H^{-1}}}  \Big )\le \frac{1}{C}|\om|^{1/2}\|f\|^{1/2}_{H^{-1}}
\end{equation}
Now, either
$$
\|u_{\eps}\|_{H^1(\om)}\le \frac{1}{kC}\,\|f\|_{H^{-1}}
$$
or
\begin{equation}
\nonumber
\Big (\|u_{\eps}\|_{H^1}-\frac{1}{kC}{\|f\|_{H^{-1}}}  \Big )^{3/2}\le \frac{1}{C}|\om|^{1/2}\|f\|^{1/2}_{H^{-1}}
\end{equation}
In both cases, we have that \eqref{apriori} holds.
\end{proof}

\begin{remark}
We stress that for $|\omega_{\eps}|\to 0$ the constant $C$ appearing in inequalities \eqref{estint} can be chosen independent of $\eps$ (see the discussion following equation \eqref{poinC} in the Appendix); hence, also the constant in the estimate \eqref{apriori} is independent of $\eps$.
\end{remark}
\begin{remark}
\label{lpbounds}
By the above estimate and by the previously mentioned Sobolev imbeddings, we obtain a priori bounds of the solutions in $L^p(\om)$, with $p\le \frac{2N}{N-2}$ if $N\ge 3$ and for every $p\ge 1$ if $N=2$.
\end{remark}

One easily verifies that the bound \eqref{apriori} holds for the potential $U$ of the unperturbed problem \eqref{probu} with $k=C=1$. We now prove additional properties
of $U$ which will be useful in the sequel.
%special case of theorem $2.4.2.7$ in \cite{Grisvard}:
%\begin{proposition}
%\label{regw2p}
%Let $\om$ be a bounded open subset of $\R^N$ with a
%$\mathcal{C}^{1,1}$ boundary $\partial\om$. Then, for every $F\in L^p(\om)$ and
%every $g\in W^{1-\frac{1}{p},\,p}(\partial\om)$ with $\int_{\om}F=\int_{\partial\om}g$ and $p\in (1,+\infty)$ there exists a unique (up to a constant)
%$u\in W^{2,\,p}(\om)$ which is a solution of
%\begin{eqnarray}
%  \nonumber \Delta u(x) &=& F(x)\quad {\rm in}\,\, \om\\
%\nonumber  \partial_{\nu}u(x)&=& g(x)\quad {\rm on}\,\, \partial \om
%\end{eqnarray}
%\end{proposition}
%By the above proposition, by  remark \ref{lpbounds} and by the bound \eqref{apriori} (with $u_{\eps}$ replaced by $U$) we readily get

\begin{proposition}
\label{regU}
Let $\om$ be a bounded domain in $\R^N$ with $\partial\Omega\in C^{1,1}$ and let $f\in L^p(\om)$ for any $p>2$ if $N=2$ and for $p>3$ if $N=3$; then the (unique) weak  solution $U$ of
\begin{equation}
\label{probu1}
\left\{
  \begin{array}{ll}
    -\Delta U+U^3=f & \hbox{in $\om$} \\
    \displaystyle{\frac{\partial U}{\partial\mathbf{n}}}=0, & \hbox{on $\partial\om$},
  \end{array}
\right.
\end{equation}
%belongs to $W^{2,\,p}(\om)$ for every $p> 1 $ if $N=2$ and for $1<p\le 2$ if $N=3$.
satisfies
\begin{equation}\label{gradU}
\| U\|_{L^{\infty}(\Omega)},\|\nabla U\|_{L^{\infty}(\Omega)}\leq C(\|f\|_{L^p(\om)}+\|f\|^3_{L^p(\om)})
\end{equation}
\end{proposition}
\begin{proof}
By the previous remark, $U^3\in L^p(\Omega)$ for every $p\ge 1$ if $N=2$ and for $1\le p\le 2$ if $N=3$; by
the equation in \eqref{probu1} the same holds (in the weak sense) for $\Delta U$. Hence, we can apply known regularity results for the Neumann problem (see e.g., Theorem $2.4.2.7$ in \cite{Grisvard})
to conclude that $U$ belongs to $W^{2,\,p}(\om)$ for every $p> 1 $ if $N=2$ and for $1<p\le 2$ if $N=3$,  with
\[
\|U\|_{W^{2,\,p}(\om)}\leq C(\|f\|_{L^p(\om)}+\|f\|^3_{L^p(\om)})
\]
Now, it is known that $W^{2,\,p}(\om)\subset \mathcal{C}^k(\overline\om)$ for $k=\big [2-N/p\big]$ (see \cite{brezis}, section 9.3); hence, in the case $N=2$ it follows that $U\in \mathcal{C}^1(\overline\om)$ whenever the datum $f$ in \eqref{probu1} satisfies $f\in L^p(\om)$ with $p>2$ and
\begin{equation}
\nonumber
\| U\|_{L^{\infty}(\Omega)},\|\nabla U\|_{L^{\infty}(\Omega)}\leq C(\|f\|_{L^p(\om)}+\|f\|^3_{L^p(\om)})
\end{equation}
 In the case $N=3$, one  obtains that $U$ is H\"{o}lder continuous on $\om$; nevertheless, the same $\mathcal{C}^1$ regularity can be readily achieved by repeated application of the previous arguments  since $U^3,f\in L^p(\Omega)$ with $p>3$ and hence $U\in W^{2,p}(\Omega)$ for $p>3$.
 %In the same way, one can prove that more regularity of the datum $f$ (on a smooth domain $\om$) implies more regularity of the solution $U$.
\end{proof}
Let us now recall that for $f\ge 0$ we have $U\ge 0$ (see remark \ref{positive}); furthermore, we have a comparison principle:
\begin{proposition}
Let $f_2\ge f_1$ satisfy the assumptions of Proposition \ref{regU} and let $U_1$, $U_2$ be the  solutions to \eqref{probu1} with $f=f_1$ and $f=f_2$ respectively. Then,
$U_2\ge U_1$ in $\om$.
\end{proposition}
\begin{proof}
The function $W=U_2-U_1$ solves the problem
\begin{equation}
\label{compar}
\left\{
  \begin{array}{ll}
    -\Delta W=-QW+f_2-f_1 & \hbox{in $\om$} \\
    \displaystyle{\frac{\partial W}{\partial\mathbf{n}}}=0, & \hbox{on $\partial\om$},
  \end{array}
\right.
\end{equation}
where $Q=U_1^2+U_1U_2+U_2^2\ge 0$. Let $\om^-=\{x\in\overline\om\,|\, W(x) <0$\}; since $W$ is continuous in $\overline\om$, the set $\om^-$ is open. Moreover, by the above equation, $W$ is \emph{superharmonic} in $\om^-$ and therefore it assumes the minimum value at some point on the boundary $\partial\om^-$. On the other hand, such point must belong to
$\overline{\partial\om^-\backslash\partial\om}$ due to the homogeneus Neumann condition and to the Hopf principle. But $W=0$ on this set, so that $\om^-=\emptyset$ and $W\ge 0$ in $\om$.
\end{proof}

\begin{corollary}
Assume that $\mathrm{essinf }_{x\in\om}f(x)=m$. Then, the solution $U$ to problem \eqref{probu1} satisfies
\begin{equation}
\label{lowbdU}
U(x)\ge m^{1/3}, \quad\quad x\in \om
\end{equation}
\end{corollary}
\begin{proof}
Apply Proposition \ref{compar} by choosing $f_1=m$ and $f_2=f$. Since $U_1=m^{1/3}$, the above bound follows.
\end{proof}
%if, in addition, $f$ is continuous with $f>0$ in $\om$ we also get $U>0$ in $\om$, since otherwise $U$ would have be an %interior minimum point where $\Delta U<0$.

\smallskip
Let us now discuss the regularity of the solution $u_{\eps}$
 we first note that, by remark \ref{lpbounds}, the term
$$
\chi_{\om\setminus\ome}u_\eps^3
$$
is bounded in $L^p(\om)$  for $1<p\leq 2$ if $N=3$ and in $L^p(\om)$ for any $p\ge 1$ if $N=2$. On the other hand, $u_{\eps}$ satisfies
\begin{equation}
\nonumber
 -\div(k_\eps(x)\nabla u_\eps)=f-\chi_{\om\setminus\ome}u_\eps^3,\quad \quad x\in\om
 \end{equation}
with $k_{\eps}$ defined by \eqref{defkesp}. Since $f\in L^p(\Omega)$ with $p>1$ if $N=2$ and $p>3$ if $N=3$   we can apply the interior estimate in \cite{GT} (Theorem $8.24$) which yields, for
any $\om'\subset\subset\om$
\begin{equation}
\label{regueps}
\|u_{\eps}\|_{\mathcal{C}^{0,\alpha}(\overline\om')}\le C\big (\|u_{\eps}\|_{L^2(\om)}+\|u_{\eps}\|^3_{L^6(\om)}+ \|f\|_{L^{p}(\Omega)}\big)\leq C\big (\|u_{\eps}\|^3_{H^1(\om)}+ \|f\|_{L^{p}(\Omega)}\big)
\end{equation}
where $0<\alpha<1$, $C>0$ only depend on $N$, $k$, $p$ and $\om'$.
%\todo [Manca qualcosa qui ...?]

%$$K\leq \frac{1}{k}\Big (\|f\|_{L^{p}}+\|u_{\eps}\|_{L^{3/2}}  \Big).$$
Finally, using \eqref{apriori}, we obtain
\begin{equation}
\label{regueps2}
\|u_{\eps}\|_{\mathcal{C}^{0,\alpha}(\overline\om')}\le C
\end{equation}
where $C$ depends only on $\Omega',k,N$ and on $ \|f\|_{L^{p}(\Omega)}$.
Now, by taking $\om'\supset\ome$ and by observing that $k_{\eps}=1$ in $\om\backslash\om'$, it is not difficult to show that $u_{\eps}$ is \emph{uniformly H\"{o}lder continuous in} $\overline\om$ and that
 \begin{equation}
\label{regueps3}
\|u_{\eps}\|_{\mathcal{C}^{0,\alpha}(\overline\Omega)}\le C
\end{equation}
where $C$ depends only on $\Omega,k,N$ and on $ \|f\|_{L^{p}(\Omega)}$.
%(obviously, by choosing $f$ in $L^p$ or smooth, the solution is even more regular in $\om\backslash{\overline\omega}_{\eps}$).

\subsection{Estimate on the $H^1$ norm of $u_\eps-U$}

\begin{theorem}
\label{h1estimweps}
\label{ext} Let $f\in L^p(\Omega)$ for some $p>N$ ($N=2,\, 3)$; assume further that $f\ge m>0$ a.e. in $\om$. Let $U$ be the solution to problem \eqref{probu} and $u_\eps$ the solution to problem \eqref{probeps}. Then

\begin{equation}\label{ineqh1estimweps}
\| u_\eps-U\|_{H^1(\Omega)}\leq C |\ome|^{\frac{1}{2}}
\end{equation}

where $C$ is a positive constant that depends on $k,\,\om,\,m$ and on $ \|f\|_{L^p(\om)}$.
\end{theorem}

\begin{proof}
Using \eqref{probu}, we obtain

\begin{equation}\nonumber
-\Delta U=-\div\left(k_\eps \nabla U\right)-\div\left((1-k_\eps) \nabla U\right)=-\div\left(k_\eps \nabla U\right)-(1-k)\div\left(\chi_{\ome} \nabla U\right)=-U^3+f
\end{equation}

and therefore

\begin{equation}\label{scompu}
-\div\left(k_\eps \nabla U\right)+\chi_{\om\setminus\ome}U^3= f+(1-k)\div\left(\chi_{\ome} \nabla U\right)-\chi_{\ome}U^3
\end{equation}

Now subtracting the above \eqref{scompu} from the equation for $u_\eps$ in \eqref{probeps} we get

\begin{equation}\nonumber
-\div\left(k_\eps \nabla (u_\eps-U)\right)+\chi_{\om\setminus\ome}(u_\eps^3-U^3)= -(1-k)\div\left(\chi_{\ome} \nabla U\right)+\chi_{\ome}U^3
\end{equation}

that, letting $w_\eps=u_\eps-U$ and $q_\eps=U^2+Uu_\eps+u_\eps^2$, we can rewrite as

\begin{equation}\label{equaw}
-\div\left(k_\eps \nabla w_\eps\right)+\chi_{\om\setminus\ome}w_\eps q_\eps= (k-1)\div\left(\chi_{\ome} \nabla U\right)+\chi_{\ome}U^3
\end{equation}

Let's now observe that in order to prove the theorem it is enough to show that

\begin{equation}\label{ineqgrad}
\|\nabla w_\eps\|_{L^2}\leq C |\ome|^{\frac{1}{2}}
\end{equation}

This follows from the fact that we can write $w_\eps=\tilde{w}_\eps +a_\eps$, where

\begin{equation}\label{defwtildea}
\int_{\om\setminus\ome}\tilde{w}_\eps q_\eps =0 \ \ \ \ \text{and} \ \ \ \ a_\eps=
\frac{1}{\int_{\om\setminus\ome}q_\eps}\,\int_{\om\setminus\ome}{w}_\eps q_\eps
\end{equation}

For the function $\tilde{w}_\eps$ we have by Poincar\'e inequality (see Appendix)

\begin{equation}\label{poincarwtilde}
\left\|\tilde{w}_\eps\right\|_{L^2} \leq C \left\|\nabla \tilde{w}_\eps\right\|_{L^2}\,\,\big (=C \left\|\nabla {w}_\eps\right\|_{L^2}\big )
\end{equation}

Moreover, being $$\displaystyle{\int_\om \div\left(k_\eps \nabla w_\eps\right)=\int_{\partial\om}\frac{\partial w_\eps}{\partial{\bf n}}=0}$$ from \eqref{probu} and $$\displaystyle{\int_{\om}\div\left(\chi_{\ome} \nabla U\right)=0}$$ from divergence theorem, using \eqref{equaw} and integrating over $\om$  we get

\begin{equation}
|a_\eps|=\frac{1}{\int_{\om\setminus\ome}q_\eps}\left|\int_{\om\setminus\ome}{w}_\eps q_\eps\right|=
\frac{1}{\int_{\om\setminus\ome}q_\eps}\left|\int_{\ome}U^3\right|
\end{equation}
Now, by our assumptions on $f$, by the elementary estimate $q_\eps\ge \frac{3}{4}U^2$ and
by \eqref{lowbdU}, we readily obtain
\begin{equation}
\label{estaeps}
|a_\eps|\leq \frac{4}{3 m^{2/3}|\om\setminus\ome|}\|U\|^3_{L^6(\om)}\,|\ome|^{\frac{1}{2}}
\end{equation}

Then, using \eqref{gradU}, \eqref{poincarwtilde} and \eqref{estaeps},
\begin{equation}\label{ineqgrad1}
\|w_\eps\|_{H^1}=\|\tilde{w}_\eps +a_\eps\|_{H^1}\leq\|\tilde{w}_\eps\|_{H^1}+\left| a_\eps\right| |\om|^{\frac{1}{2}}\leq
 C |\ome|^{\frac{1}{2}}
\end{equation}

In order to prove \eqref{ineqgrad},
%\todo[cambiare eqref]
multiplying \eqref{equaw} times $w_\eps$ and integrating over $\om$ by parts, we get

\begin{equation}\nonumber
\int_\om k_\eps \left|\nabla w_\eps\right|^2+\int_{\om\setminus\ome}w_\eps^2 q_\eps=-(k-1)\int_{\ome} \nabla U\nabla w_\eps+\int_{\ome}U^3 w_\eps
\end{equation}

which leads to

\begin{equation}\nonumber
k\|\nabla w_\eps\|_{L^2}^2\leq\left|(1-k)\int_{\ome} \nabla U\nabla w_\eps\right|+\left|\int_{\ome}U^3 w_\eps\right|\leq
\end{equation}

\begin{equation}\label{estimw1}
\left\{(1-k)\|\nabla U\|_{L^\infty(\ome)}\|\nabla w_\eps\|_{L^2}+\|U\|_{L^\infty(\ome)}^3 \|w_\eps\|_{L^2}\right\}\left|\ome\right|^{\frac{1}{2}}.
\end{equation}

using again the decomposition \eqref{defwtildea}, Poincar\'e inequality \eqref{poincarwtilde} for $\tilde{w}_\eps$, estimate \eqref{estaeps} for $a_\eps$ and \eqref{gradU} we obtain

\begin{equation}\nonumber
%k\|\nabla w_\eps\|_{L^2}^2\leq\left\{\big[(1-k)C_1+C_2\big] \|\nabla w_\eps\|_{L^2}+C_2 C_3 \left|\ome\right|^{\frac{1}{2}}\right\}\left|\ome\right|^{\frac{1}{2}}.
k\|\nabla w_\eps\|_{L^2}^2\leq C\left(\|f\|_{L^p(\om)}+\|f\|^3_{L^p(\om)}\right) \left\{\|\nabla w_\eps\|_{L^2}+\left|\ome\right|^{\frac{1}{2}}\right\}\left|\ome\right|^{\frac{1}{2}}.
\end{equation}

where $C=C(k,\Omega)$.

Finally, solving second order inequality, we get

\begin{equation}\nonumber
%\|\nabla w_\eps\|_{L^2}\leq\frac{\big[(1-k)C_1+C_2\big]\left|\ome\right|^{\frac{1}{2}}+\sqrt{\big[(1-k)C_1+C_2\big]^2\left|\ome\right|+k C_2 C_3 \left|\ome\right|}}{2k}= C_4\left|\ome\right|^{\frac{1}{2}}
\|\nabla w_\eps\|_{L^2}\leq C \left|\ome\right|^{\frac{1}{2}}
\end{equation}
where $C$ is a positive constant depending on $\Omega,k$ and on $\|f\|_{L^p(\om)}$.

\end{proof}

We now derive energy estimates for $u_{\epsilon}-U$.

\begin{theorem}
\label{L2estimweps}

Let $f$ satisfy the same assumptions as in theorem \ref{h1estimweps}. Then
\begin{equation}\label{energy2}
\|u_{\eps}-u\|_{L^2(\om)}\leq C|\omega_{\eps}|^{\frac{1}{2}+\eta}
\end{equation}
 for some $\eta>0$ and  where $C$ is a positive constant depending on $k,\Omega, m$ and on $\|f\|_{L^{p}(\om)}$.

\end{theorem}
\begin{proof}
%[Proof of Theorem \ref{L2estimweps}]
Set $w_{\eps}=u_{\eps}-U$. Then, $w_{\eps}\in H^1(\om)$ satisfies
\begin{equation}\label{en11}
\int_{\om}
\nabla w_{\eps}\cdot\nabla \phi\, dx+\int_{\om\backslash\omega_{\eps}}
q_{\eps}w_{\eps} \phi\, dx=(k-1)\int_{\omega_{\eps}}\nabla u_{\eps}\cdot\nabla \phi\, dx+\int_{\omega_{\eps}}U^3\phi\, dx\quad\forall \phi\in H^1(\om).
\end{equation}
where $q_{\eps}=u_{\eps}^2+u_{\eps}^2U^2+U^2$ and, by the estimate \eqref{ineqh1estimweps},
\[
\|w_{\eps}\|_{H^1(\om)}\leq C|\omega_{\eps}|^{1/2}
\]
where $C=C(k,\Omega,\|f\|_{L^p(\om)})$.
Consider now $\bar w_{\eps}\in H^1(\om)$ weak solution to
\begin{equation}\label{en2}
\int_{\om}
\nabla \bar w_{\eps}\cdot\nabla \phi\, dx+\int_{\om\backslash\omega_{\eps}}
q_{\eps}\bar w_{\eps} \phi\, dx=\int_{\om} w_{\eps}\phi\, dx\quad\forall \phi\in H^1(\om).
\end{equation}
Then, choosing $\phi=\bar w_{\eps}$,  one  has
\[
\|\bar w_{\eps}\|_{H^1(\om)}\leq C\|w_{\eps}\|_{H^1(\om)}\leq C|\omega_{\eps}|^{1/2} .
\]
Furthermore  by  Theorem $2.4.2.7$ in \cite{Grisvard} we have that $\bar w_{\eps}$ belongs to $H^{2}$
\begin{equation}\label{reg}
\|\bar w_{\eps}\|_{H^2(\om)}\leq C\|w_{\eps}\|_{L^2(\om)}
\end{equation}
 Choosing $\phi=w_{\eps}$ into $(\ref{en2})$ we get
\begin{equation}\label{en3}
\int_{\om}
\nabla \bar w_{\eps}\cdot\nabla w_{\eps}, dx+\int_{\om\backslash\omega_{\eps}}
q_{\eps}\bar w_{\eps} w_{\eps}\, dx=\int_{\om}w^2_{\eps}\, dx.
\end{equation}
On the other hand, choosing $\phi=\bar w_{\eps}$ into (\ref{en11}) we derive
\begin{equation}\label{en4}
\int_{\om}
\nabla \bar w_{\eps}\cdot\nabla w_{\eps}, dx+\int_{\om\backslash\omega_{\eps}}
q_{\eps}\bar w_{\eps} w_{\eps}\, dx=(k-1)\int_{\omega_{\eps}}\nabla u_{\eps}\cdot\nabla \bar w_{\eps}\, dx+\int_{\omega_{\eps}}U^3\bar w_{\eps}\, dx.
\end{equation}
Hence, by (\ref{en3}) and (\ref{en4}), we have,
\begin{equation}\label{en5}
\int_{\om}w^2_{\eps}\, dx=(k-1)\int_{\omega_{\eps}}\nabla u_{\eps}\cdot\nabla \bar w_{\eps}\, dx+\int_{\omega_{\eps}}U^3\bar w_{\eps}\, dx.
\end{equation}
From (\ref{reg}) and Sobolev Imbedding Theorem we have that $\bar w_{\eps}\in W^{1,p'}(\omega_{\eps})$ for any $p'>1$ if $N=2$ and for $1<p'\leq 6$ if $N=3$.
\begin{equation}\label{regrad}
\|\bar w_{\eps}\|_{W^{1,p'}(\Omega)}\leq C \|w_{\eps}\|_{L^2(\om)}.
\end{equation}
Since $U\in H^1(\om)$, again from Sobolev Imbedding Theorem, $U^3\in L^p$, for all  $p> 1 $ if $N=2$ and for $1<p\leq 2$ if $N=3$ . Hence, applying Holder inequality and choosing $p'$ so that $1<p<2$, we get
\[
\int_{\om}w^2_{\eps}\, dx\leq |k-1|\|\nabla \bar w_{\eps}\|_{L^{p'}(\omega_{\eps})}\|\nabla u_{\eps}\|_{L^p(\omega_{\eps})}+\| \bar w_{\eps}\|_{L^{p'}(\omega_{\eps})}\|U^3\|_{L^p(\omega_{\eps})}
\]
Observe now that
\[
\|\nabla u_{\eps}\|_{L^p(\omega_{\eps})}\leq \|\nabla w_{\eps}\|_{L^p(\omega_{\eps})}+\|\nabla U\|_{L^p(\omega_{\eps})}
\]
%By   the energy estimate  (\ref{ineqh1estimweps})  and by (\ref{gradU})
By \eqref{gradU} the second term can be bounded as follows
\[
\|\nabla U\|_{L^p(\omega_{\eps})}\leq C(\|f\|_{L^p(\Omega)}+\|f\|^3_{L^p(\Omega)})|\omega_{\eps}|^{1/p}
\]
where $C=C(\Omega)$.
Moreover, by H\"older inequality and by the energy estimates  (\ref{ineqh1estimweps}) we have
\[
\|\nabla w_{\eps}\|_{L^p(\omega_{\eps})}\leq |\ome|^{\frac{1}{p}-\frac{1}{2}}\|w_{\eps}\|_{H^1}\le C |\omega_{\eps}|^{1/p}
\]
Hence, we get the bound
\[
\|\nabla u_{\eps}\|_{L^p(\omega_{\eps})}\leq C |\omega_{\eps}|^{1/p}
\]
where $C=C(\|f\|_{L^p(\Omega)},k,m)$. Analogously
\[
\|U^3\|_{L^p(\omega_{\eps})}\leq C|\omega_{\eps}|^{1/p}
\]
where $C=C(\Omega, \|f\|_{L^p(\Omega)})$.  Recalling (\ref{regrad}), we  get
\[
\int_{\om}w_{\eps}^2\leq C\|w_{\eps}\|_{L^2(\om)}|\omega_{\eps}|^{1/p}
\]
which finally gives
\[
\|w_{\eps}\|_{L^2(\om)}\leq C|\omega_{\eps}|^{1/p}
\]
with $\frac{1}{p}>\frac{1}{2}$ and $C=C(k,(\|f\|_{L^p(\Omega)})$.
\end{proof}

\section{Proof of main result: the asymptotic formula}

In this section we deduce an asymptotic representation formula for the perturbed potential
$$w_{\eps}=u_{\eps}-U$$
analogous to the one obtained in theorem $1$ of \cite{capvoge} for a voltage perturbation in the presence of inhomogeneities.

Let $N_U(x,y)$ be the Green function of the operator $-\Delta+3U^2$ with \emph{homogeneous} Neumann condition defined in \eqref{greenU}.

Note that we can write
\begin{equation}
\nonumber
N_U(x,y)=N(x,y)+z(x,y)
\end{equation}
where $N$ is the \emph{Neumann function} for the Laplacian, satisfying
\begin{equation}\label{greenN}
-\Delta_x N(x,y)=\delta(x-y)\quad\mathrm{for}\,\, x\in\om,\quad\quad\quad \frac{\partial N}{\partial n_x}\Big |_{\partial\om}=\frac{1}{|\partial\om|}
\end{equation}
and, for every $y\in\om$, the function $x\mapsto z(x,y)$ solves the problem
\begin{equation}
\label{probz}
\left\{
  \begin{array}{ll}
    -\Delta_x z(x,y)+3U^2(x) z(x,y)=-3U^2(x)N(x,y), & \hbox{in $\om$} \\
    \displaystyle{\frac{\partial z}{\partial {n_x}}}=-\frac{1}{|\partial\om|}, & \hbox{on $\partial\om$},
  \end{array}
\right.
\end{equation}
We recall that $N(\cdot,y)\in W^{1,1}(\om)$ and therefore it belongs to $L^p(\om)$ for $p$ in some interval depending on the dimension  (for every $p>1$ in dimension two). Then, by the smoothness of $U$ and by the same regularity arguments as in the previous section, we may take $z$ and $\nabla z$ continuous and bounded; it follows in particular that
\begin{equation}
\nonumber
%\label{stimNU}
N_U(\cdot,y)\in L^p(\om)
\end{equation}

Let us now multiply both the equations \eqref{probeps} and \eqref{probu} by a test function $\phi$, integrate by parts and use the boundary condition; we get the identity
\begin{equation}
\label{prelid}
\int_\om k_\eps\nabla u_\eps\nabla\phi + \int_{\om\setminus\ome}u_\eps^3\phi=\int_\om \nabla U\nabla\phi + \int_{\om}U^3\phi
\end{equation}
By subtracting to both sides of \eqref{prelid} the quantity
\begin{equation}
\nonumber
\int_{\om}\nabla u_\eps\nabla\phi+\int_{\om}u_\eps^3\phi
\end{equation}
we obtain
\begin{equation}
\nonumber
%\label{prelid2}
\int_{\ome} (k-1)\nabla u_\eps\nabla\phi - \int_{\ome}u_\eps^3\phi=\int_\om \nabla (U-u_\eps)\nabla\phi + \int_{\om}(U^3-u_\eps^3)\phi
\end{equation}
By introducing the perturbed potential $w_\eps=u_\eps-U$, we can write the above equation in the form
\begin{equation}
\nonumber
\int_\om \nabla w_\eps\nabla\phi + \int_{\om}w_\eps\big (U^2+U\,u_\eps+u_\eps^2\big )\phi=\int_{\ome} (1-k)\nabla u_\eps\nabla\phi + \int_{\ome}u_\eps^3\phi
\end{equation}
Finally, by using the identity
$$U^2+U\,u_\eps+u_\eps^2=3U^2+3U\,w_\eps+w_\eps^2$$
we get
\begin{equation}
\label{asyform1}
\int_\om \nabla w_\eps\nabla\phi + \int_{\om}3U^2\,w_\eps\phi=\int_{\ome} (1-k)\nabla u_\eps\nabla\phi + \int_{\ome}u_\eps^3\phi-\int_{\om}3U\,w_\eps^2\phi -\int_{\om}\,w_\eps^3\phi
\end{equation}
Let us fix $y\in\partial\om$ (or even $y\in \om\backslash \overline\ome$) and let $\phi_m\in \mathcal{C}^1(\om)$ be a sequence converging to $N_U(\cdot,y)$
in $W^{1,1}(\om)$ and in $\mathcal{C}^1(\overline D)$, where $\ome\subset D\subset\subset\om$ . Now, the regularity of $U$ provided by \eqref{gradU} and the discussion following \eqref{regueps} allow us to insert $\phi_m$ into \eqref{asyform1} and to pass to the limit, so that
\begin{eqnarray}
\nonumber
  \int_\om \nabla w_\eps\nabla_x N_U\,dx + \int_{\om}3U^2\,w_\eps N_U\,dx &= &\int_{\ome} (1-k)\nabla u_\eps\nabla_x N_U\,dx + \int_{\ome}u_\eps^3N_U\,dx  \\
\nonumber   &-& \int_{\om}3U\,w_\eps^2N_U\,dx -\int_{\om}\,w_\eps^3N_U\,dx
\end{eqnarray}
After integration by parts in the first term by using \eqref{greenU} (here we exploit the \emph{homogeneous Neumann condition} satisfied by $N_U$) we obtain
\begin{equation}
\label{asyform3}
w_\eps(y)=\int_{\ome} (1-k)\nabla u_\eps\nabla_x N_U\,dx + \int_{\ome}u_\eps^3N_U\,dx
-\int_{\om}3U\,w_\eps^2N_U\,dx -\int_{\om}\,w_\eps^3N_U\,dx
\end{equation}
The following result is a first step towards an asymptotic representation formula in our non linear setting:
\begin{proposition}
\label{firststep}
Let $\mathbf{1}_{\ome}$ denotes the indicator function of the set $\ome$. Then the following relation holds
\begin{equation}
\label{asyform4}
w_\eps(y)=|\ome|\Bigg ((1-k)\int_{\om} |\ome|^{-1}\mathbf{1}_{\ome}\nabla u_\eps\nabla_x N_U\,dx +
\int_{\om}|\ome|^{-1}\mathbf{1}_{\ome}u_\eps^3 N_U\,dx\Bigg ) +o(|\ome|)
\end{equation}
\end{proposition}
\begin{proof} We need to prove suitable bounds of the two last terms in \eqref{asyform3}. By H\"{o}lder inequality, the last term is bounded by
$\|w_\eps\|^3_{L^{3q}(\om)}\|N_U\|_{L^{p}(\om)}$,
where $q=p/(p-1)$. Hence, by Sobolev embedding and by
\eqref{ineqh1estimweps} we get
\begin{equation}
\nonumber
\Big |\int_{\om}\,w_\eps^3N_U\,dx\Big |\le C |\ome|^{3/2}
\end{equation}
for some constant $C$ depending on $k$, $\om$ and $U$. Let us now consider the remaining term; by the boundedness of $U$ and again by H\"{o}lder inequality we have
\begin{equation}
\label{stimweps2}
\Big |\int_{\om}3U\,w_\eps^2N_U\,dx\Big |\le 3\|U\|_{L^{\infty}(\om)}\Big |\int_{\om}\,w_\eps^2N_U\,dx\Big |\le
 3\|U\|_{L^{\infty}(\om)}\|N_U\|_{L^{p}(\om)}\|w_\eps\|^2_{L^{2q}(\om)}
\end{equation}
By a version of the \emph{Gagliardo-Nirenberg inequality} for bounded domains (the constants depending only on $q$ and on the domain, see \cite{nire}) we now get
\begin{equation}
\nonumber
\|w_\eps\|_{L^{2q}(\om)}\le C_1\|\nabla w_\eps\|^{1-\frac{1}{q}}_{L^{2}(\om)}\,\|w_\eps\|^{\frac{1}{q}}_{L^{2}(\om)} +C_2\|w_\eps\|_{L^{2}(\om)}\le
\end{equation}
\centerline{(by \eqref{ineqh1estimweps} and
\eqref{energy2})}
\begin{equation}
\nonumber
\le \tilde C_1 \,|\ome|^{\frac{1}{2}+\frac{\eta}{q}}+\tilde C_2 \,|\ome|^{\frac{1}{2}+\eta}
\end{equation}
Then, the proposition follows by inserting these estimates into \eqref{asyform3}.
\end{proof}

\begin{remark}
Equation \eqref{asyform4} should be compared with the analogous formula $(8)$ given in \cite{capvoge} for the steady state voltage perturbation caused by internal conductivity inhomogeneities. The different sign of the term containing the gradients is due to the definition \eqref{greenU} of the Green function $N_U$.
\end{remark}
Following \cite{capvoge} we now introduce the variational solutions $V^{(j)}$, $v_{\eps}^{(j)}$ to the problems

\begin{equation}
\label{probVj}
\left\{
  \begin{array}{ll}
    \Delta V^{(j)}=0, & \hbox{in $\om$} \\
    \displaystyle{\frac{\partial V^{(j)}}{\partial\mathbf{n}}}=n_j, & \hbox{on $\partial\om$},
  \end{array}
\right.
\end{equation}

\begin{equation}
\label{probvjeps}
\left\{
  \begin{array}{ll}
    \div(k_\eps(x)\nabla v_{\eps}^{(j)})=0, & \hbox{in $\om$} \\
    \displaystyle{\frac{\partial v_{\eps}^{(j)}}{\partial\mathbf{n}}}=n_j, & \hbox{on $\partial\om$},
  \end{array}
\right.
\end{equation}
$n_j$ being the $j-$th  coordinate of the outward normal to $\partial\om$ and where the functions $V^{(j)}$, $v_{\eps}^{(j)}$ are normalized by $\int_{\partial\om}V^{(j)}=\int_{\partial\om}v_{\eps}^{(j)}=0$. We observe that
\begin{equation}
\label{esplVj}
V^{(j)}=x_j-\frac{1}{|\partial\om|}\int_{\partial\om}x_j
\end{equation}
and that the difference $v_{\eps}^{(j)}-V^{(j)}$ satisfies estimates analogous to \eqref{ineqh1estimweps} and to \eqref{energy2} (see \cite{capvoge} sect.2). Hence, by integration by parts and by exploiting such estimates, we get (see \cite{capvoge} sect.3, eqs. (20)-(21))
\begin{equation}
\label{capvogeform}
\int_\om k_\eps\nabla (u_\eps-U)\nabla(v_{\eps}^{(j)}\phi)\,dx =\int_\om \nabla (u_\eps-U)\nabla(V^{(j)}\phi)\,dx
+\int_{\ome} (k-1)\nabla (u_\eps-U)\nabla\phi\,V^{(j)}\,dx + o(|\ome|)
\end{equation}
for every $\phi$ smooth enough. Now, again using the weak form of the equations \eqref{probeps} and \eqref{probu}, we easily get the identities
\begin{equation}
\nonumber
\int_\om k_\eps\nabla (u_\eps-U)\nabla(v_{\eps}^{(j)}\phi)\,dx =\int_{\ome}(1-k) \nabla U \nabla(v_{\eps}^{(j)}\phi)\,dx
+\int_{\ome} U^3 \,v_{\eps}^{(j)}\phi\,dx+\int_{\om\backslash\ome}(U^3-u^3_{\eps})\,v_{\eps}^{(j)}\phi\,dx
\end{equation}
\begin{equation}
\nonumber
\int_\om \nabla (u_\eps-U)\nabla(V^{(j)}\phi)\,dx=
\int_{\ome} (1-k)\nabla u_\eps\nabla(V^{(j)}\phi)\,dx +\int_{\ome} U^3\, V^{(j)}\phi\,dx+\int_{\om\backslash\ome}(U^3-u^3_{\eps})\,V^{(j)}\phi\,dx
\end{equation}
By inserting these into \eqref{capvogeform} we obtain
\begin{equation}
\nonumber
\int_{\ome}(1-k) \nabla U \nabla(v_{\eps}^{(j)}\phi)\,dx
+\int_{\ome} U^3 \,v_{\eps}^{(j)}\phi\,dx+\int_{\om\backslash\ome}(U^3-u^3_{\eps})\,v_{\eps}^{(j)}\phi\,dx=
\end{equation}
\begin{equation}
\nonumber
\int_{\ome} (1-k)\nabla u_\eps\nabla(V^{(j)}\phi)\,dx +\int_{\ome} U^3\, V^{(j)}\phi\,dx+\int_{\om\backslash\ome}(U^3-u^3_{\eps})\,V^{(j)}\phi\,dx
+\int_{\ome} (k-1)\nabla (u_\eps-U)\nabla\phi\,V^{(j)}\,dx + o(|\ome|)
\end{equation}
That is, by straightforward rearrangements,
\begin{equation}
\nonumber
(1-k)\int_{\ome} \nabla U \nabla(v_{\eps}^{(j)}\phi)\,dx=(1-k)\Bigg [\int_{\ome} \nabla u_\eps\nabla(V^{(j)}\phi)\,dx
-\int_{\ome} \nabla u_\eps\nabla\phi\,V^{(j)}\,dx + \int_{\ome} \nabla U\nabla\phi\,V^{(j)}\,dx\Bigg ]
\end{equation}
\begin{equation}
\nonumber
-\int_{\ome} U^3\, (v_{\eps}^{(j)}-V^{(j)})\phi\,dx+\int_{\om\backslash\ome}(u^3_{\eps}-U^3)\,(v_{\eps}^{(j)}-V^{(j)})\phi\,dx
 + o(|\ome|)
\end{equation}
By the boundedness of $U$, $u_{\eps}$, by H\"{o}lder inequality and by the previous $L^2$ estimates of the perturbations $u_{\eps}-U$ and $v_{\eps}^{(j)}-V^{(j)}$, we conclude that \emph{the whole last term of the above equation is} $o(|\ome|)$.
Hence we can write
\begin{equation}
\nonumber
\int_{\ome} \nabla U \nabla(v_{\eps}^{(j)}\phi)\,dx=\int_{\ome} \nabla u_\eps\nabla(V^{(j)}\phi)\,dx
-\int_{\ome} \nabla u_\eps\nabla\phi\,V^{(j)}\,dx + \int_{\ome} \nabla U\nabla\phi\,V^{(j)}\,dx+ o(|\ome|)
\end{equation}
\begin{equation}
\nonumber
=\int_{\ome} \nabla u_\eps\nabla V^{(j)}\,\phi\,dx
 + \int_{\ome} \nabla U\nabla\phi\,V^{(j)}\,dx+ o(|\ome|)
\end{equation}
\begin{equation}
\nonumber
=\int_{\ome} \nabla u_\eps\nabla V^{(j)}\,\phi\,dx
 + \int_{\ome} \nabla U\nabla\phi\,v_{\eps}^{(j)}\,dx+O\big (\|v_{\eps}^{(j)}- V^{(j)}\|_{L^2(\om)}|\ome|^{1/2}\|\nabla U \|_{L^{\infty}(\ome)} \big ) +o(|\ome|)
\end{equation}
\begin{equation}
\nonumber
=\int_{\ome} \nabla u_\eps\nabla V^{(j)}\,\phi\,dx
 + \int_{\ome} \nabla U\nabla\phi\,v_{\eps}^{(j)}\,dx+o(|\ome|)
\end{equation}
After a further rearrangement, we get
\begin{equation}
\nonumber
\int_{\ome} \nabla U \nabla v_{\eps}^{(j)}\,\phi\,dx=\int_{\ome} \nabla u_\eps\nabla V^{(j)}\,\phi\,dx
 +o(|\ome|)
\end{equation}
A final rescaling yields
\begin{equation}
\label{capvoge22}
\int_{\om} \nabla U |\ome|^{-1}\mathbf{1}_{\ome}\nabla v_{\eps}^{(j)}\,\phi\,dx=\int_{\om} |\ome|^{-1}\mathbf{1}_{\ome}\nabla u_\eps\,\nabla V^{(j)}\,\phi\,dx+o(1)
\end{equation}
By the results of \cite{capvoge} there exist a regular Borel measure $\mu$, functions $\mathcal{M}_{i\,j}\in L^2(\om,d\mu)$
and a sequence ${\omega}_{\eps_n}$, with $|{\omega}_{\eps_n}|\to 0$, such that
\begin{equation}
\label{misure}
|{\omega}_{\eps_n}|^{-1}\mathbf{1}_{{\omega}_{\eps_n}}\,dx\rightarrow d\mu,\quad\quad\quad |{\omega}_{\eps_n}|^{-1}\mathbf{1}_{{\omega}_{\eps_n}}\frac{\partial}{\partial x_i} v_{\eps_n}^{(j)}\,dx\rightarrow
\mathcal{M}_{i\,j}\,d\mu
\end{equation}
in the weak* topology of the dual of $\mathcal{C}^0(\overline\om)$. Then, passing to the limit in \eqref{capvoge22} and by recalling \eqref{esplVj} we can state
\begin{proposition}
\label{secondstep}
Let $u_{\eps}$, $U$ denote the solutions to \eqref{probeps} and \eqref{probu} and let $\omega_{\eps_n}$, with $|{\omega}_{\eps_n}|\to 0$, be a sequence satisfying \eqref{sepbound} and \eqref{misure}. Then
\begin{equation}
\label{limisure}
\lim_{n\to\infty} |{\omega}_{\eps_n}|^{-1}\mathbf{1}_{{\omega}_{\eps_n}}\frac{\partial u_{\eps_n}}{\partial x_j} \,dx=
\mathcal{M}_{i\,j}\frac{\partial U}{\partial x_i}\,d\mu
\end{equation}
in the weak* topology of the dual of $\mathcal{C}^0(\overline\om)$.
\end{proposition}
We are now in position to prove our asymptotic representation formula. We'll state it here in a more precise way:
\begin{theorem}
\label{asyreprr}
Let $u_{\eps}$, $U$ denote the solutions to \eqref{probeps} and \eqref{probu} and let $\omega_{\eps_n}$, with $|{\omega}_{\eps_n}|\to 0$, be a sequence satisfying \eqref{sepbound} and \eqref{misure}. Then, if $w_{\eps_n}=u_{\eps_n}-U$, we have
\begin{equation}
\label{asyformfinall}
w_{\eps_n}(y)=|\omen|\int_{\om} \Big ((1-k)\mathcal{M}_{i\,j}\frac{\partial U}{\partial x_i}\frac{\partial N_U}{\partial x_j}+U^3 N_U\Big )\,d\mu(x) +o(|\omen|)\quad\quad y\in\partial \om
\end{equation}
where $N_U(x,y)$ is the solution of \eqref{greenU}.
\end{theorem}
\begin{proof}
By proposition \ref{firststep} we have
\begin{equation}
\nonumber
w_{\eps_n}(y)=|\omen|\Big (\int_{\om}(1-k) |\omen|^{-1}\mathbf{1}_{\omen}\nabla u_{\eps_n}\nabla_x N_U\,dx +
\int_{\om}|\omen|^{-1}\mathbf{1}_{\omen}u_{\eps_n}^3 N_U\,dx\Big ) +o(|\omen|)
\end{equation}
Let $K_0$ is a compact set such that $\omega_{\epsilon}\subset K_0\subset\Omega$. By the properties of $N_U$  we can find a vector valued test function $\Phi_y\in \mathcal{C}^0(\overline\om)$ such that
\[
\Phi_y(x)=\nabla_x N_U(x,y),\,\ \text{ for }x\in K_0,\,\, y\in\partial\Omega.
\]
Then, by proposition \ref{secondstep},
\begin{equation}
\nonumber
\int_{\om}(1-k) |\omen|^{-1}\mathbf{1}_{\omen}\frac{\partial u_{\eps_n}}{\partial x_j} \frac{\partial N_U}{\partial {x_j}}\,dx =\int_{\om}(1-k)\mathcal{M}_{i\,j}\frac{\partial U}{\partial x_i}\frac{\partial N_U}{\partial x_j}\,d\mu(x)+o(1)
\end{equation}
Moreover, by now standard estimates one can readily prove
\begin{equation}
\nonumber
\int_{\om}|\omen|^{-1}\mathbf{1}_{\omen}u_{\eps_n}^3 N_U\,dx=\int_{\om} U^3N_U\,d\mu(x) + o(1)
\end{equation}
By inserting the above relations in the previous identity, the theorem follows.
\end{proof}
We are now ready for:

\smallskip\noindent
\emph{Proof of theorem \ref{asyrepr}.} The asymptotic formula \eqref{asyformfinal} is the same as equation \eqref{asyformfinall} proved in the previous theorem. In order to prove the last statement of the theorem, we remark that the polarization tensor $\mathcal{M}_{i\,j}$ is defined exactly as in \cite{capvoge}; hence, the stated properties follow by the same arguments
as in section $4$ of \cite{capvoge} with trivial modifications.

\section{Localization of small inhomogeneities}

Let us consider the case of a finite number of well separated homogeneities of small diameter $\eps$ centered at points $z_1,...,z_m\in\om$. In the limit $\eps\to 0$, one obtains from the asymptotic formula \eqref{asyformfinal} (see also \cite{capvoge})
\begin{equation}
\label{asyform5}
w_\eps(y)=\eps^N\sum_{l=1}^m \Big ((1-k) \mathcal{M}_{i\,j}(z_l)\frac{\partial U}{\partial x_i}(z_l)\frac{\partial N_U}{\partial x_j}(z_l,y)+U^3 (z_l)N_U(z_l,y)\Big ) +o(\eps^N)\quad\quad y\in\partial \om
\end{equation}
We now show that the above formula, together with a suitable integration of (measured) boundary data, allows one to obtain useful identities for localizing the inhomogeneities and reconstructing the polarization tensor.

We first prove an auxiliary identity; let $g$ be a given function on $\partial\om$ and consider the (unique) solution $W$ of the boundary value problem
\begin{equation}
\label{probw}
\left\{
  \begin{array}{ll}
    -\Delta W+3U^2 W=0, & \hbox{in $\om$} \\
    \displaystyle{\frac{\partial W}{\partial\mathbf{n}}}=g, & \hbox{on $\partial\om$},
  \end{array}
\right.
\end{equation}
where $U$ is the background potential which solves \eqref{probu}. Then we have
\begin{equation}
\label{reprW}
W(z)=\int_{\partial\om}N_U(z,y)\,g(y)\,dS_y\quad\quad\quad z\in\om
\end{equation}
where $N_U$ is the Neumann function defined by \eqref{greenU}. The proof follows readily by observing that, due to the homogeneous Neumann condition satisfied by $N_U$, we can write
\begin{equation}
\nonumber
\int_{\partial\om}N_U(z,y)\,g(y)\,dS_y=
\int_{\partial\om}N_U(z,y)\,\displaystyle{\frac{\partial W}{\partial\mathbf{n}}}\,dS_y - \displaystyle{\frac{\partial N_U}{\partial\mathbf{n}}}(z,y)\,W(y)\,dS_y
\end{equation}
We now consider the \emph{average measurement}
\begin{equation}
\label{avemeas}
\Gamma\equiv \int_{\partial\om}w_\eps(y)\,g(y)\,dS_y
\end{equation}
By inserting \eqref{asyform5} in this expression and taking account of \eqref{reprW} we get
\begin{equation}
\label{useform}
\Gamma=\eps^N\sum_{l=1}^m\Big [(1-k)\mathcal{M}_{i\,j}(z_l)\frac{\partial U}{\partial x_i}(z_l)\frac{\partial W}{\partial x_j}(z_l)+U^3 (z_l)W(z_l)\Big ]+ o(\eps^N)
\end{equation}
We first apply the previous formula to the simple case of approximating the location and the polari\-zation tensor of a \emph{single} small inhomogeneity (in two dimension) centered at the point $(\bar x,\bar y)$.

We observe that by choosing a constant datum $f$ in problem \eqref{probu}, the (unique) solution is a constant background potential  $U=\lambda\equiv f^{1/3}$. In that case, the equation for the auxiliary function $W$ becomes
\begin{equation}
\nonumber
-\Delta W(x,y)+3\lambda^2 W(x,y)=0
\end{equation}
The above equation has a family of solutions of the form
\begin{equation}
\nonumber
W(x,y)= e^{ax+by},\quad\quad\quad a,\,b\,\in \R
\end{equation}
provided that $a^2+b^2=3\lambda^2$. In particular, we have the two solutions
$$W_1(x,y)=e^{\lambda\sqrt 3 \,x},\quad\quad W_2(x,y)=e^{\lambda\sqrt 3 \,y}$$
respectively with Neumann data
$$g_1(x,y)=\lambda\sqrt 3 \,n_1(x,y)\,e^{\lambda\sqrt 3 \,x},\quad\quad
g_2(x,y)=\lambda\sqrt 3 \,n_2(x,y)\,e^{\lambda\sqrt 3 \,y},\quad\quad (x,y)\in\partial\om$$
where $n_i$, $i=1,2$ are the component of the  normal unit vector to $\partial\om$.

Now, by denoting with $\Gamma_1(\lambda)$, $\Gamma_2(\lambda)$ the average measurements \eqref{avemeas} with data $g_1$, $g_2$, and with $w_{\eps}=u_{\eps}-\lambda$ where $u_{\eps}$ solves \eqref{probeps} with $f=\lambda^3$, we obtain from \eqref{useform} (with $N=2$, $m=1$)
\begin{equation}
\label{eqposiz}
\Gamma_1(\lambda)=\eps^2\lambda^3\,e^{\lambda\sqrt 3 \,\bar x}+o(\eps^2),\quad\quad
\Gamma_2(\lambda)=\eps^2\lambda^3\,e^{\lambda\sqrt 3 \,\bar y}+o(\eps^2)
\end{equation}
By choosing a specific value of $\lambda$, the above relations can be used to approximate the position of the center of a small inhomogeneity.
The determination of the polarization tensor requires a non constant background potential $U$. In order to further simplify the problem, we assume that $\om=[0,1]\times[0,1]$ and try to identify the single element $M_{11}$ of the ($2\times 2$) polarization matrix (we also assume that $k$ is known). By the geometry of the domain, we can take a background potential $U=U(x)$ independent of $y$,  provided that $U'(0)=U'(1)=0$. Hence, we look for an auxiliary function $W=W(x)$ which solves the linear ordinary equation of the second order
\begin{equation}
\nonumber
-W''(x)+3U(x)^2 W(x)=0
\end{equation}
By looking for a solution in the form
$$W(x)=e^{\varphi(x)}$$
we find that the function $\varphi$ satisfies the equation
$$\varphi''(x)+\varphi'(x)^2=3U(x)^2$$
By the substitution $\psi(x)=\varphi'(x)$ we are reduced to a first order \emph{Riccati equation}
\begin{equation}
\label{riccati}
\psi'(x)+\psi(x)^2=3U(x)^2
\end{equation}
In general, there are no explicit solutions of such equation for a given $U$; on the other hand, there are large families of functions $\psi$ such that the left hand side of  \eqref{riccati} is a positive function with vanishing derivative at $x=0$ and $x=1$. Thus, for any such $\psi$, the function $U$ \emph{defined} (except for the sign) by \eqref{riccati} is an admissible background potential.
For example, a straightforward calculation shows that
$$\psi(x)=\frac{1}{3}(x^2+x-3)$$
solves \eqref{riccati} with $U(x)$ smooth function in $[0,1]$ satisfying homogeneous Neumann conditions (it can be easily seen that no linear $\psi$ can generate an admissible non constant potential). Then, if we have previously detected the position $(\bar x,\bar y)$, the matrix element $M_{11}(\bar x,\bar y)$ can be approximated as follows:

insert at the right hand side of \eqref{useform} the values $U(\bar x)$, $U'(\bar x)$ calculated with the above potential, together with $W(\bar x)=e^{\int\psi\,(\bar x)}$, $W'(\bar x)=\psi(\bar x)e^{\int\psi\,(\bar x)}$; note that by an appropriate choice of the integration constant we may take $e^{\int\psi\,(\bar x)}=1$. Then,
put at the left hand side of \eqref{useform} the average measurement \eqref{avemeas} with $g$ the Neumann datum of $W(x)$ and $w_{\eps}=u_{\eps}-U$, $u_{\eps}=u_{\eps}(x,y)$ being the solution of \eqref{probeps} with $f(x,y)=-U''(x)+U(x)^3$.
\begin{remark}
It may be interesting to compare the above discussion to the detection of one small inhomogeneity for the linear problem in [Ammari-Moskow-Vogelius]. We stress that the reconstruction algorithm for the non linear problem, though more difficult from a computational point of view, allows to detect the position of the inhomogeneity (by using a constant background potential) independently of the polarization tensor. However, it is not clear if it is possible to perform an efficient localization of many separated inhomogeneities.
\end{remark}

\section{Appendix: Poincar\'e inequalities}

There are different versions of inequalities which are usually known as Poincar\'e inequalities. Essentially, they relate the $L^2$ norm of the fluctuation of a function to the $L^2$ norm of its gradient. In this paper we use the following special case of the inequality proved in \cite{liebloss}, Theorem 8.11 :

\begin{theorem}
\label{poincare}
Let $g$ be a function in $L^2(\om)$  such that $\int_{\om}g=1$. Then, there is $S>0$ which depends on $\om$, $g$, such that for any $u\in H^1(\om)$
\begin{equation}
\label{liebloss}
\Big \|u-\int_{\om}ug\Big \|_{L^2(\om)}\le S \|\nabla u\|_{L^2(\om)}
\end{equation}
\end{theorem}
The proof follows a classical {\it reductio ad absurdum} argument relying on compactness.

If we now choose $u\equiv u_{\eps}$,
\begin{equation}
\label{defg}
g=|\om\backslash\ome|^{-1}\mathbf{1}_{\om\backslash\ome}
\end{equation}
and put
$$\bar u_{\eps}=|\om\backslash\ome|^{-1}\int_{\om\backslash\ome}u_{\eps}$$
we obtain
\begin{equation}
\label{stimpoinc}
\|u_{\eps}\|^2_{L^2(\om)}\le 2\big (\|u_{\eps}-\bar u_{\eps}\|^2_{L^2(\om)}+|\om|\bar u_{\eps}^2\big )\le
2S^2 \|\nabla u_{\eps}\|^2_{L^2(\om)}+2\frac{|\om|}{|\om\backslash\ome|}\|u_{\eps}\|^2_{L^2(\om\backslash\ome)}
\end{equation}
By this estimate it follows easily
\begin{equation}
\label{poinC}
\|u_{\eps}\|^2_{H^1(\om)}\le \frac{1}{C}\Big (\|\nabla u_{\eps}\|^2_{L^2(\om)}+ \|u_{\eps}\|^2_{L^2(\om\setminus\ome)}\Big )
\end{equation}
which was used in theorem \ref{exist} and in proposition \ref{mainest}. Since the functions \eqref{defg} are uniformly bounded for $\eps\to 0$, one can show that the costant $S$ can be chosen independent of $\eps$; thus, by  \eqref{stimpoinc}, we can also take $C$ independent of $\eps$ in \eqref{poinC}.

Finally, by choosing  $u\equiv{w}_{\eps}=u_{\eps}-U $,
$$g=\Big (\int_{\om\setminus\ome} q_\eps\Big )^{-1} \,q_{\eps}\,\mathbf{1}_{\om\backslash\ome}$$
(where $q_{\eps}=U^2+U u_{\eps}+u^2_{\eps}$)
and by recalling \eqref{defwtildea}, we readily see that \eqref{liebloss} is equivalent to the estimate \eqref{poincarwtilde}.

\end{document}